\documentclass[11pt,a4paper]{amsart}
\usepackage[utf8]{inputenc}
\usepackage[english]{babel}
\usepackage[left=2cm,right=2cm,top=2cm,bottom=2cm]{geometry}
\usepackage{graphicx}
\usepackage{amssymb}
\usepackage{amsmath}
\usepackage{amsthm}
\usepackage{algorithm}
\usepackage{algorithmic}
\usepackage{comment}
\usepackage{pdfpages}
\usepackage{appendix}
\usepackage{stmaryrd}
\usepackage{tikz}
\usepackage{bbm}
\usepackage[maxbibnames=50]{biblatex}
\usepackage{csquotes}
\usepackage{hyperref}
\hypersetup{
  colorlinks   = true, 
  linkcolor    = black,
  citecolor    = blue      
}

\providecommand{\ed}{\mathrm e}

\providecommand{\Fl}{\mathrm{Fl}}
\providecommand{\Gl}{\mathrm{GL}}

\numberwithin{equation}{section}
\newtheorem{theorem}{Theorem}[section]
\newtheorem{proposition}[theorem]{Proposition}
\newtheorem{corollary}[theorem]{Corollary}
\newtheorem{lemma}[theorem]{Lemma}
\newtheorem{definition}[theorem]{Definition}

\newtheorem{remark}[theorem]{Remark}

\title{Enumeration of crossings in two-step puzzles}
\author
{Quentin François}
\address
{Quentin François: CEREMADE, CNRS, Université Paris-Dauphine, Université PSL, 75016 Paris, France
\& DMA, École normale supérieure, Université PSL, CNRS, 75005 Paris, France.}

\email{\href{mailto:quentin.francois@dauphine.psl.eu}{quentin.francois@dauphine.psl.eu}}

\begin{document}

\begin{abstract}
    We prove a formula which gives the number of occurrences of certain labels and local configurations inside two-step puzzles 
    introduced by Buch, Kresch, Purbhoo and Tamvakis from the work of Knutson.
    Puzzles are tilings of the triangular lattice by edge labeled tiles
    and are known to compute the Schubert structure constants of the cohomology of two-step flag varieties. 
    The formula that we obtain depends only on the boundary conditions of the puzzle. 
    The proof is based on the study of color maps which are tilings of the triangular lattice by edge labeled tiles 
    obtained from puzzles.
\end{abstract}

\maketitle


\section{Introduction}

In his article \cite{Knutson_conjecture}, Knutson conjectured that the structure constants of the cohomology ring 
of a partial flag variety $\Gl(n) / P$ 
can be computed by the number of tilings of the triangular lattice called puzzles using specific tiles with side labels. The conjecture was first proved in 
the case of the Grassmannian variety \cite{Knutson_Tao_1}, \cite{Knutson_Tao_2}. 
The puzzle rule was extended in \cite{BKT_gromov_witten_inv_on_grass} to compute the Gromov-Witten invariants which are structure constants 
for the small quantum cohomology ring of the Grassmannians. 
Gromov-Witten invariants are particular instances of structure constants of the two-step variety $\Fl(a, b, n)$, see the work of Buch \cite{Buch_2003}
using the kernel and span of rational curves. 
The puzzle conjecture for the two-step flag variety was eventually proved in \cite{puzzle_conj_two_step}. 
 The proof is based on an analogue of the jeu de taquin algorithm where local configurations are propagated in puzzles according to specific rules.  
Another combinatorial expression of the structure coefficients for the two-step flag variety had been proved 
by Coskun \cite{Coskun_Mondrian} using diagrams called Mondrian tableaux. 
The extension of the puzzle rule to equivariant Schubert structure constants for the two-step flag variety 
was conjectured by Coskun and Vakil \cite{Coskun_Vakil} 
and proved by Buch  \cite{Buch_mutations}. In the latter, the author introduced transformations on 
equivariant two-step puzzles called mutations which in particular encompass the local rules of \cite{puzzle_conj_two_step}.

Edge labels on boundaries of two-step puzzles \cite{puzzle_conj_two_step} can be either $0$, $1$ or $2$. 
At the scale of the whole puzzle, the labels $0$ and $1$ create lines starting from the boundaries 
and crossing each other inside the puzzle. There are two possible types of crossings up to rotations. 
In one type of crossing, the lines joining identical labels from both sides cross each other by keeping their direction constant which 
is encoded in the puzzle by the label $7$ inside the configuration. In the other type of crossing, the line joining labels $0$ may not 
keep its direction constant which is encoded by the presence of at least one label $3$ in the configuration.  
 In \cite{françois2024positiveformulaproductconjugacy}, inspired from the hive model by Knutson and Tao in \cite{Knutson_Tao_1} and \cite{Knutson_Tao_2},
Tarrago and the author constructed a bijection between two-step puzzles and objects called dual two-colored hives which involve tilings
of the triangular lattice called color maps together with edge labels satisfying inequality and equality conditions. 
The bijection converts labels $7$ in two-step puzzles to edges of color $m$ in color maps. Moreover, the number of crossings of the second type 
is equal to the number of edges of color $3$ in the color map.

The present paper gives a formula for the number of crossings of each type that is, for both the number of labels $7$ and the number of crossings
of the second type in any two-step puzzle of $\Fl(a, b, n)$. The formula depends only on the $012$ labels on the puzzle boundary. 
This paper is organised as follows. In Section \ref{sec:main_result} we give the necessary definitions to state the main result. 
The latter is first expressed in terms of color maps in Theorem \ref{th:nb_hard_crossings} which translates to crossings 
in two-step puzzles in Corollary \ref{cor:nb_cross_puzzle}. 
Section \ref{sec:arrows} recalls some definitions of local configurations in color maps useful for the rest of the paper. 
Section \ref{sec:g_zero_case} proves the main identity in a special case where the boundaries of the color map are in a simple form. 
Section \ref{sec:general_case} starts with local propagations of configurations in color maps called 
gashes which are directly inspired from \cite{Buch_mutations} and proves the main identity by induction using propagations.

\subsection*{Acknowledgement} The author is supported by the Agence Nationale de la Recherche funding 
CORTIPOM ANR-21-CE40-0019. We thank Pierre Tarrago and Anders Skovsted Buch 
for the fruitful discussions on puzzles and related notions.

\section{Main result}
\label{sec:main_result}

\begin{definition}[Triangular lattice]
    Let $n \geq 1$ and let $\xi = \ed^{\frac{i\pi}{3}}$. Denote $T_n = \{ r + s \xi, 0 \leq r + s \leq n \} $ 
    the vertices of the triangular lattice of size $n$ and $E_n = \{ (x, x+v): x,x+v \in T_n \text{ and } v \in \{ -\xi^{2l}, 0 \leq l \leq 2 \} \}$ 
    the set of edges in $T_n$. The faces of the lattice $T_n$ are triangles which are called direct (respectively reversed) 
    if the corresponding vertices $(x_1, x_2, x_3) \in T_n^3$ can be labeled in such a way that $x_2 - x_1 = (1,0)$ and 
    $x_3 - x_1 = \xi$ (respectively $x_3 - x_1 = \overline{\xi}$). 
\end{definition}

\noindent
Edges in $E_n$ can only have three possible orientations. If $x = r + s \xi \in T_n $, we define three coordinates $(x_0, x_1, x_2)$ by
    \begin{equation*}
        x_0 = n - (r+s), \ x_1 = r \text{ and } x_2 = s.
    \end{equation*}
    
\begin{definition}[Edge coordinate and type]
    We say that an edge $e = (x, x+v)$ is of type $l$ for $l \in \{ 0,1,2 \}$ when $v =  -\xi^{2l}$. 
    The origin of $e$ is $x$ and the coordinates of $e$ is the triple $(e_0, e_1, e_2) = (x_0, x_1, x_2)$. 
    The height of $e$ of type $l$ is $h(e) = e_l$. Define also the boundary edges of $E_n$ by
    \begin{align*}
        \partial_0^{(n)} &=  \left( ((n-r+1,0), (n-r,0)), 1 \leq r \leq n \right) \\
        \partial_1^{(n)} &=  \left( ( n\xi + (r-1) \overline{\xi}, (n\xi + r \overline{\xi}) ) , 1 \leq r \leq n \right) \\
        \partial_2^{(n)} &=  \left( ((r-1) \overline{\xi},  r \overline{\xi}) , 1 \leq r \leq n \right). 
    \end{align*}
\end{definition}

\begin{definition}[Color map]
    Let $n \geq 1$. A color map is a map $C: E_n \rightarrow \{ 0, 1 ,3, m \}$ such that the boundary colors around each 
    triangular face in the clockwise order is either $(0, 0, 0)$, $(1, 1, 1)$, $(1, 0, 3)$ or $(0, 1, m)$ up to a cyclic rotation. 
\end{definition}

\noindent
The values of a color map $C$ on the boundary edges are denoted $ \partial C = (\partial_0 C, \partial_1 C, \partial_2 C)$ and are defined 
for $l \in \{ 0,1,2 \}$ as $\partial_l C = C_{\vert \partial_l^{(n)}}$. We say that $C$ has boundary condition 
$ \partial = (\partial_0, \partial_1, \partial_2)$ if $\partial C = \partial$. \\
\noindent
Alternatively, one can view a color map $C$ as a tiling of $T_n$ by the set of edge labeled tiles of 
Figure \ref{fig:pieces_dual_hive} where tiles can be rotated. 
The last two tiles are respectively called $3$ and $m$ lozenges in accordance with the color of their middle edge.

\begin{figure}[H]
    \centering
    \includegraphics[scale=0.8]{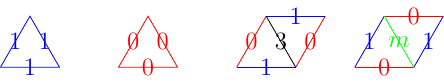}
    \caption{Possible tiles for color maps}
    \label{fig:pieces_dual_hive}
\end{figure}

\noindent
As there is an equal number of both $0$ and $1$ labels on each side of two-step puzzles, we will consider boundary conditions 
$\partial C \in \{ 0, 1 \}^{3n}$ having an equal number of $0$ and $1$ colored edges 
respectively denoted by $n_0$ and $n_1$ so that $n_0 + n_1 = n$, see Figure \ref{fig:ex_color_map} below. 
Such boundary conditions correspond to those of two-step puzzles \cite{puzzle_conj_two_step} where one 
removed the labels $2$ from the boundary $012$ strings.

\begin{figure}[H]
    \centering
    \includegraphics[scale=1]{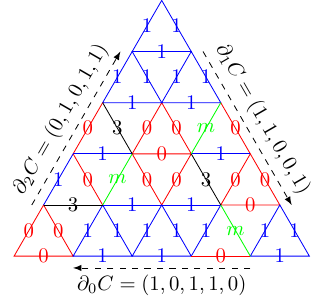}
    \caption{A color map on $E_5$ with boundary condition $ \partial C = ((1,0,1,1,0), (1,1,0,0,1), (0,1,0,1,1))$.}
    \label{fig:ex_color_map}
\end{figure}

\begin{definition}[Gash numbers]
\label{def:gash_number}
    Let $C: E_n \rightarrow \{ 0,1,3,m \}$ be a color map. For any $l \in \{ 0,1,2 \} $ and edge $e \in \partial_l^{(n)}$ denote by 
    $ n(C, e) = | \{ e' \in \partial_l^{(n)} : h(e') < h(e) \text{ and } C(e') = 1 \}|$ the number of $1$ colored edges east 
    (respectively north, south) to $e$ if $e \in \partial_0^{(n)}$ (respectively $e \in \partial_1^{(n)}$, $e \in \partial_2^{(n)}$). 
    The gash numbers of the color map $C$ are defined for $l \in \{ 0,1,2 \} $ as
    \begin{equation}
        G(C, l) = \sum_{e \in \partial_l^{(n)}: C(e)=0} n(C, e).
    \end{equation}
\end{definition}

\noindent
For instance, in the color map $C$ of Figure \ref{fig:ex_color_map}, one has $G(C, 0) = 4$, $G(C, 1) = 4$, $G(C, 2) = 1$. 
The main result of this paper is the following count of the number of $3$ and $m$ colored edges in color maps 
which depends only on the gash numbers. 

\begin{theorem}[Label count in color maps]
\label{th:nb_hard_crossings}
    Let $C$ be a color map on $E_n$ having $n_0$, respectively $n_1$, edges of color $0$, respectively $1$, on each of its boundaries.
     Let $m(C)$ and $s(C)$ denote respectively the number of $m$ and $3$ colored edges in $C$. Then, 
    \begin{equation}
    \label{eq:hc_count_general}
        m(C) = G(C, 0) + G(C, 1) + G(C, 2) - n_0 n_1
    \end{equation}
    and
    \begin{equation}
    \label{eq:sc_count_general}
        s(C) = 2n_0 n_1 - G(C, 0) - G(C, 1) - G(C, 2).
    \end{equation}
\end{theorem}

\noindent
Let us also mention that one can count other types of tiles in $C$. 
The enumeration of triangular faces having edges of the same color is given in Corollary \ref{cor:triangular_faces_count}. \\
\\
In \cite[Theorem 5.3]{françois2024positiveformulaproductconjugacy} a bijection was defined between two-step puzzles of \cite{puzzle_conj_two_step}
and objects called two colored dual hives consisting of a color map together with a label map, 
see \cite[Definition 5.1]{françois2024positiveformulaproductconjugacy} for details on the definition. 
In particular, this bijection converts labels $7$ of two-step puzzles into to $m$ colored edges. 
Moreover the number of edges $e \in E_n$ with color $C(e)=3$ in the obtained color map  
is equal to the number of pieces of the form of Figure \ref{fig:soft_crossing} where the number of labels $2$ is arbitrary and where 
the configuration can be rotated. This piece is one of the composed puzzle pieces of \cite{puzzle_conj_two_step} which we
call a \textit{soft crossing} in the rest of this paper.
\begin{figure}[ht]
    \centering
    \includegraphics[scale=1]{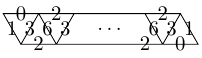}
    \caption{A soft crossing in two step puzzles.}
    \label{fig:soft_crossing}
\end{figure}

Recall that the clockwise labels on boundaries of two-step puzzles are $012$ strings.
 Let $u$ be a $012$ string of length $n$: $u = u_1 \dots u_n$ where $u_i \in \{ 0, 1, 2 \}, 1 \leq i \leq n$.
  In accordance with Definition \ref{def:gash_number}, define 
\begin{equation*}
    G(u) = \sum_{1 \leq i \leq n: u_i = 0} |\{ j \leq i: u_j = 1 \}|.
\end{equation*}
Theorem \ref{th:nb_hard_crossings} yields a direct computation of the number of labels $7$ and soft crossings
 in any two-step puzzle given in Corollary \ref{cor:nb_cross_puzzle}.

\begin{corollary}[Labels $7$ and soft crossings in two-step puzzles]
\label{cor:nb_cross_puzzle}
    Let $P$ be a two-step puzzle with boundary given by three $012$ strings $u, v, w$ respectively on the left, right and bottom sides in clockwise order,
    each having $n_0$ symbol $0$ and $n_1$ symbol $1$. 
    Let $n(P, \mathrm{sc})$ and $n(P, 7)$ denote respectively the number of soft crossings and labels $7$ in $P$. Then,
    \begin{align*}
        n(P, 7) &= G(u) + G(v) + G(w) - n_0n_1, \\
        n(P, \mathrm{sc}) &= 2n_0n_1 - G(u) - G(v) - G(w).
    \end{align*}
\end{corollary}

\begin{proof}
    Let $C$ be the color map associated to the image of $P$ by the bijection from \cite[Definition 5.2]{françois2024positiveformulaproductconjugacy}. 
    Then, $G(C, 0) = G(w)$, $G(C, 1) = G(v)$ and $G(C, 2) = G(u)$. Moreover, $n(P, \mathrm{sc}) = s(C)$ and $n(7, P) = m(C)$ 
    from which one derives the result using \eqref{eq:hc_count_general} and \eqref{eq:sc_count_general}.  
\end{proof}

\noindent
\textbf{Sketch of the proof of Theorem \ref{th:nb_hard_crossings}.}
In Section \ref{sec:arrows}, we introduce some transformations on color maps that will play a role in the rest of the paper. 
In Section \ref{sec:g_zero_case}, we prove Theorem \ref{th:nb_hard_crossings} in the case where $G(C, 2) = 0$. 
This is done by showing that when $G(C, 2) = 0$, the color map can be reduced to a simple color map in which the counting is explicit. 
In Section \ref{sec:general_case}, we give a procedure to transform any color map $C$ to another color map $C'$ such that 
$G(C', 2) = G(C, 2) - 1$ from which one can prove Theorem \ref{th:nb_hard_crossings} by induction. 

\section{Arrows}
\label{sec:arrows}

In this section, we recall some definitions on local configurations that were introduced in \cite{françois2024positiveformulaproductconjugacy}. 

\begin{definition}[Opening]

    Let $x \in T_n$. An opening of type $l \in \{ 0,1,2 \}$ at $x$ is a pair of edges $(e, e') \in E_n^2$ such that 
    if $e = (e_1, e_2)$, $e' = (e_1', e_2')$ with $(e_1, e_1', e_2, e_2') \in T_n^4$ and $t(e), t(e')$ are the types of $e$ and $e'$,
    \begin{align*}
        e_i &= e_i' = x \text{ for some } i \in \{ 1, 2 \}, \\
        \{ t(e), t(e') \} &= \{ l-1, l+1 \} \text{ and } C(e) = C(e') \in \{ 0,1 \}. 
    \end{align*}
    The color of the opening is defined as the color of edges $e$ and $e'$.
\end{definition}

\noindent
Consider an opening $a = (e, e')$ at $x$ of type $l$ and color $c \in \{ 0,1 \}$. 
Let $e'' = e''(a)$ be the edge such that $e, e'$ are edges of the lozenge with middle edge $e''$. 
The only possible colors of the edge $e''$ are $C(e'') \in \{ 0,1 \}$. If $C(e'') = c$, the two triangular faces of the lozenge with middle edge $e''$ 
have all of their edges colored $c$. If $C(e'') \neq c$, then there is an opening $a'$ of type $l$ and color $c$ at the other endpoint of $e''$. 
Note that there can only be finitely many such openings before $C(e'') = c$.

\begin{definition}[Arrow]
    Let $a = (e, e')$ be an opening of type $l$ and color $c$. 
    Let $ r \geq 0$ be the number of successive openings having middle edge $e''$ such that $C(e'') \neq c$ with $C(e'') \in \{ 0,1 \}$ 
    as in the previous paragraph. We say that the configuration of edges consisting of the $r \geq 0$ successive pairs of $3$ and $m$ 
    lozenges together with the pair of direct and reverse faces with boundary edges of color $c$ is an arrow of length $r \geq 0$ at the opening $a$. 
\end{definition}

\noindent
See Figure \ref{fig:arrows} for examples of openings and arrows.
\begin{figure}[H]
    \centering
    \includegraphics[scale=0.7]{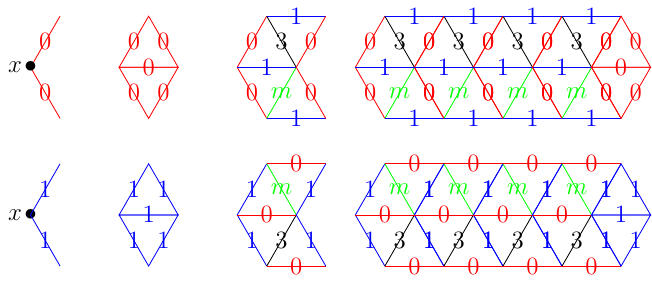}
    \caption{First row from left to right : an opening $a$ with color $0$ and type $0$ at $x$, the case $C(e'') = c$, 
    the case $C(e'') \neq c$ and an arrow of length $r=4$. The second row is the analog for color $1$. }
    \label{fig:arrows}
\end{figure}

\noindent
Let $A$ be an arrow of length $r \geq 1$ at an opening with center $x$. 
The reversal of $A$ is the configuration obtained by applying a rotation of $\pi$ to $A$. 
An example of arrow reversal is given in Figure \ref{fig:arrow_reversal}.

\begin{figure}[H]
    \centering
    \includegraphics[scale=0.8]{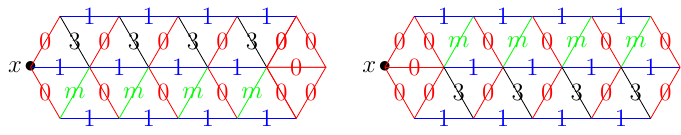}
    \caption{Reversal of an arrow of length $4$ at $x$.}
    \label{fig:arrow_reversal}
\end{figure}

\section{The case $G(C, 2) = 0$}
\label{sec:g_zero_case}

In this section, we prove \eqref{eq:hc_count_general} for color maps $C$ such that $G(C, 2) = 0$. 
We first reduce the color map $C$ to $C'$ so that all the $0$ colored edges on $\partial_0 C' $ are consecutive 
starting from the bottom left coner of $T_n$. 
This is done in Section \ref{subsec:reduction_of_partial_0}. In Section \ref{subsec:case_p_1}, we give an explicit counting of 
$m(C')$ and in Section \ref{subsec:proof_g_zero} we show \eqref{eq:hc_count_general} when $G(C, 2) = 0$ using the two previous sections.

\subsection{Reduction of color maps}
\label{subsec:reduction_of_partial_0}

\begin{definition}[Lozenge and trapeze regions]
    Let $x = (x_0, x_1, x_2) \in T_n$ and let $r, s \geq 0$ be such that $(r, s) \neq (0,0)$. Define the lozenge region $L[r, s, x] \subset E_n$ as 
    \begin{equation}
    \label{eq:lozenge_region}
    L[r, s, x] = E (\{ x + u + v \xi, (u,v) \in  \{ 0, \dots, r \} \times \{ 0, \dots, s \}  \} )
    \end{equation}
    where for a subset $V \subset T_n$, $E(V) \subset E_n$ is the subset of edges having both endpoints in $V$. 
    Moreover, the trapeze region $T[r, s, x] \subset E_n$ is defined for $s \geq r$ as
    \begin{equation}
    \label{eq:trapeze_region}
    T[r, s, x] = E (\{ x + u + v \xi, (u,v) \in  \{ 0, \dots, r \} \times \{ 0, \dots, s \}: v + u \leq s  \} ).
    \end{equation}
\end{definition}

\noindent For an illustration of lozenge and trapeze regions, see Figure \ref{fig:bottom}.

\begin{lemma}[Filling a lozenge region]
\label{lem:lozenge_filling}
    Let $L[r, s, x]$ be a region as in \eqref{eq:lozenge_region}. 
    Suppose that its boundary edges $\{ (x+u, x+u-1), 1 \leq u \leq r \}$ and $\{ (x +v \xi, x+ (v+1) \xi), 0 \leq v \leq s-1 \}$ 
    are colored $1$ and $0$ respectively. Then, every edge in $L[r, s, x]$ of type $1$ has color $3$, which determines the color of 
    all edges in $L[r, s, x]$ uniquely.
\end{lemma}

\begin{proof}
    For any $v \in T_n$ such that $C((v+1, v)) = 1$ and $C((v, v+\xi)) = 0$, there is only one possible set of values for a color map 
    $C$ on edges $ (v+\xi, v+1), (v+1+\xi, v+\xi), (v+1, v+1+\xi) $ which is given by $(3,1,0)$. 
    Applying this constraint to $v = x, x+1, \dots, x+r-1$ in this order and using induction to fill the remaining region 
    $L[r, s-1, x + \xi]$ shows the result.
\end{proof}

\begin{lemma}[Filling a trapeze region]
\label{lem:trapeze_filling}
    Let $T[r, s, x]$ be a region as in \eqref{eq:trapeze_region}. Suppose that its edges $\{ (x+u, x+u-1), 1 \leq u \leq r \}$ 
    and $\{ (x +v \xi, x+ (v+1) \xi), 0 \leq v \leq s-1 \}$ are colored $0$. 
    Then, up to some arrow reversals, every edge in $T[r, s, x]$ has color $0$.
\end{lemma}

\begin{proof}
    We will prove the result by induction over $r$. Assume that $r=1$. The triangular face having edges $(x+1, x)$ and $(x, x+ \xi)$ 
    colored $0$ has its third edge $(x+ \xi, x+1)$ of type $1$ also colored $0$. Then, the edges $(x+ \xi, x+1)$ and $(x+ \xi, x+2\xi)$ 
    form a $0$ opening that we call $o_1$. Consider the arrow $A_1$ at $o_1$ of length $\ell \geq 0$ having its other endpoint at $x + \ell + 1$. 
    Apply the arrow reversal as in Figure \ref{fig:arrow_reversal} which only changes the colors of edges inside $A_1$.
    Then, the edges $(x+1, x+1+\xi), (x+1+\xi, x+\xi)$ and $(x+2\xi, x+1+\xi)$ have color $0$. 
    Notice that in the resulting configuration, the edges $(x+ 2\xi, x+1+\xi)$ and $(x+ 2\xi, x+3\xi)$ form a $0$ opening.
    Moreover, reversing an arrow between endpoints $x$ and $x + \ell + 1$ does not modify the colors of the edges $e$ 
    having origin $y$ such that $y_0 \geq x_0$. 
    By successively considering the $0$ openings formed by edges $(x+ v \xi, x+1+ (v-1) \xi)$ and $(x+ v\xi, x+(v+1)\xi)$ 
    for $1 \leq v \leq s-1$, we get that $T[1, s, x]$ has every edge colored $0$. \\
    For $r \geq 2$, using the same argument as above shows that all edges in $T[1, s, x]$ are colored $0$. 
    Since  $T[r, s, x] = T[1, s, x] \cup T[r-1, s-1, x+1]$, one gets the result by induction. 
\end{proof}

\begin{remark}[Filling lozenge and trapeze regions]
    Note that the results of Lemmas \ref{lem:lozenge_filling} and \ref{lem:trapeze_filling} remain valid if one swaps labels $0$ and $1$, 
    replacing $3$ lozenges in a lozenge region by $m$ lozenges and $0$ colored edges in the trapeze region by $1$ colored edges.
\end{remark}

\noindent
The next Lemma shows that the bottom region adjacent to $\partial_0^{(n)} $  of a color map 
has an explicit description in terms of lozenge and trapeze regions. 
An illustration of that region is given in Figure \ref{fig:bottom} where the lozenge regions are filled with lozenges having 
middle edge of type $1$ colored $3$ and trapeze regions have all of their edges colored $0$.

\begin{lemma}[Structure above $\partial_0^{(n)}$]
\label{lem:bottom_structure}
    Let $C$ be a color map such that $G(C, 2) = 0$. 
    Then, there exists $p = p(C) \geq 1$ and
     $0 = y_0 \leq x_1 < y_1 < x_2 < \dots < y_{p-1} < x_p < y_p \leq n$
    such that up to some arrow reversals, denoting $r_i = y_i -x_i$ and $b_i = x_i - y_{i-1}$ for $1 \leq i \leq p$, the regions 
     \begin{equation*}
        L[b_1, n_0, (y_0, 0)], L[b_2, n_0-r_1, (y_1, 0)], \dots, L[b_p,  r_p , (y_p, 0)]
    \end{equation*}
    are filled by lozenges having middle edge of type $1$ colored $3$ and such that regions
    \begin{equation*}
        T[r_1, n_0, (x_1, 0)], T[r_2, n_0-r_1, (x_2, 0)], \dots, T[r_p, r_p, (x_p, 0)]
    \end{equation*}
    have their edges colored $0$.
\end{lemma}

\begin{proof}
    Set $y_0 = 0$ and by convention $C((0, -1)) = C((n+1, n)) = 1$. Define
    \begin{equation}
        p = \left| \{ 1 \leq x \leq n-1:  C((x, x-1)) = 0 \text{ and } C((x+1, x)) = 1 \} \right| \geq 1
    \end{equation}
    and for $1 \leq i \leq p$,
    \begin{align}
        x_i &= \inf \{ u \geq y_{i-1} : C( (u, u-1)) = 1 \text{ and } C( (u+1, u)) = 0  \} \\
        y_i &= \inf \{ u \geq x_i : C( (u, u-1)) = 0 \text{ and } C( (u+1, u)) = 1  \}. 
    \end{align}
    We have that $y_0 \leq x_1 < y_1 < x_2 < \dots < y_{p-1} < x_p < y_p$. Recall that $r_i = y_i -x_i$ and $b_i = x_i - y_{i-1}$ for $1 \leq i \leq p$. \\
    \\
    Since $G(C, 2) = 0$, the region $L[b_1, n_0, (0,0)]$ has its edges $\{ (x, x-1), 1 \leq x \leq b_1 \}$ and 
    $\{ (x, x+\xi), 0 \leq x \leq n_0-1 \}$ colored $1$ and $0$ respectively which implies by Lemma \ref{lem:lozenge_filling} 
    that it is filled by lozenges having middle edge of type $1$ colored $3$ except in the case where $b_1 = 0$ for which 
    $L(b_1, n_0, (0,0)) = \{ (x, x+\xi), 0 \leq x \leq n_0-1 \}$ has all of its edges on $\partial_2^{(n)}$ colored $0$.
    Remark that edges in $L(b_1, n_0, (0,0))$ with coordinate $e_1$ equal to $x_1$ are colored $0$. Therefore, the trapeze region 
    $T[r_1, n_0, (x_1, 0)]$ has its boundary edges colored $0$ as in Lemma \ref{lem:trapeze_filling} which shows that up to arrow reversals, 
    it has all of its edges colored $0$. Using Lemmas \ref{lem:trapeze_filling} and \ref{lem:lozenge_filling} successively on the regions
     \begin{equation*}
        L[b_1, n_0, (y_0, 0)], L[b_2, n_0-r_1, (y_1, 0)] \dots, L[b_p,  r_p , (y_{p-1}, 0)]
    \end{equation*}
    and 
    \begin{equation*}
        T[r_1, n_0, (x_1, 0)], T[r_2, n_0-r_1, (x_2, 0)] \dots, T[r_p, r_p, (x_p, 0)]
    \end{equation*}
    gives the result. Notice that the order of the applications of Lemma \ref{lem:trapeze_filling} is compatible with the 
    arrow reversals involved for the trapeze regions in the sense that arrow reversals in $T[a, b, x]$ only affect edges $e \in E_n$ 
    such that $e_1 \geq x_1$.  
    \begin{figure}[H]
        \centering
        \includegraphics[scale=0.75]{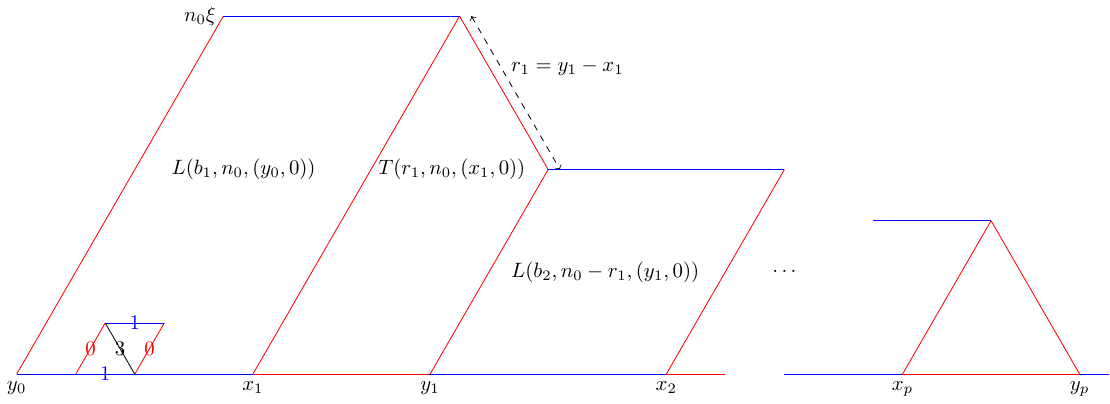}
        \caption{Region at the bottom of $C$. Lozenge regions are filled with $3$ lozenges and trapeze regions are filled with $0$ colored edges.}
        \label{fig:bottom}
    \end{figure}
\end{proof}

\noindent
Recall that for $1 \leq i \leq p, r_i(C) = y_i - x_i$ denotes the number of edges of $\partial_0 C$ which are in the $i$-th 
trapeze region of Figure \ref{fig:bottom}.

\begin{lemma}[Grouping columns]
\label{lem:grouping_col}
    Let $C$ be a color map and let $p = p(C)$ and $0 = y_0 \leq x_1 < y_1 < x_2 < \dots < y_{p-1} < x_p < y_p \leq n$ 
    be defined as in Lemma \ref{lem:bottom_structure}. Assume that $p(C) \geq 2$. Using arrow reversal, 
    adding $r_p \times b_p$ $m$-colored edges and removing $r_p \times b_p$ $3$-colored edges, 
    one can map $C$ to $C'$ such that $p' = p(C') = p(C) - 1$ and $r_{p'}(C') = r_p(C) + r_{p-1}(C)$.
\end{lemma}

\begin{proof}

Let us first define a local transformation. Consider any vertex $v \in T_n$ such that $C((v + \xi -1, v)) = 3 $ and $ C((v+1, v)) = 0$,
 so that $C((v+\xi, v+1)) = 0$. Consider the color map $C_v$ where $C_v((v, v-1)) = 0$ and $C_v((v, v + \xi)) = m$. 
 We call $C \mapsto C_v$ the replacement at $v$, see Figure \ref{fig:replacement}. 
 The color map $C_v$ has one less $3$ colored edge and one more $m$ colored edge than $C$. 

\begin{figure}[H]
    \centering
    \includegraphics[scale=0.9]{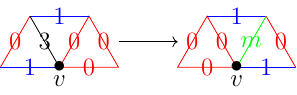}
    \caption{Replacement at vertex $v$.}
    \label{fig:replacement}
\end{figure}

\noindent
Consider the lozenge region $L[b_p+2, r_p, (y_{p-1}-1, 0)]$
where $b_p = x_p - y_{p-1}$ and $r_p = y_p - x_p \geq 1$, see Figure \ref{fig:group_col_1}. 
Apply replacements successively at $x_p, \dots, y_{p-1}+1$ and call $C_1$ the resulting color map, 
see Figure \ref{fig:group_col_2} and Figure \ref{fig:group_col_3} for an illustration of this step. 
Notice that $C_1$ has an arrow of length $b_p$ at the $0$ opening at $v_1 = (y_{p-1}, 0) + \xi$. 
Reverting this arrow creates an arrow at $v_2 = v_1 + \xi$ of length $b_p$, see Figure \ref{fig:group_col_4}. 
By reverting arrows with openings at $(y_{p-1}, 0) + q \xi, 1 \leq q \leq r_p $ each with length $b_p$,
 the resulting color map $\varphi(C)$ satisfies
\begin{align}
    p(\varphi(C)) &= p \text{ and } r_p(\varphi(C)) = r_p - 1 \text{ if } r_p \geq 2, \\
    p(\varphi(C)) &= p-1 \text{ and } r_{p(\varphi(C))}(\varphi(C)) = r_{p-1} + 1 \text{ if } r_p = 1,
\end{align}
see Figure \ref{fig:group_col_5}.
By applying the previous transformation $C \mapsto \varphi(C)$ a number of times equal to $r_p$, 
one gets a color map $C'$ such that $p(C') = p-1 $ and $r_{p(C')}(C') = r_{p-1} + r_p$. 


\begin{figure} [H]
    \begin{minipage}[c]{0.33\linewidth}
    \includegraphics[scale=0.6]{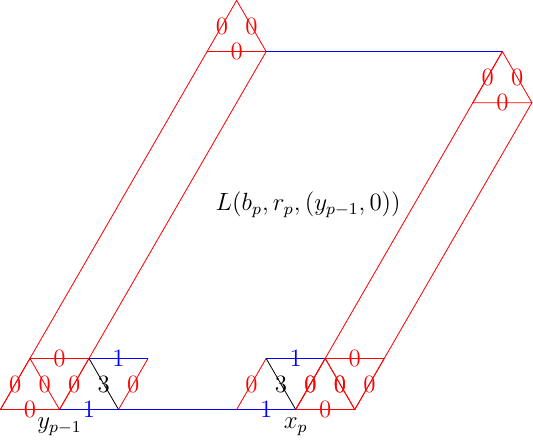}
    \caption{Initial color map $C$. \\}
    \label{fig:group_col_1}
    \end{minipage}
    \begin{minipage}[c]{0.33\linewidth}
    \includegraphics[scale=0.6]{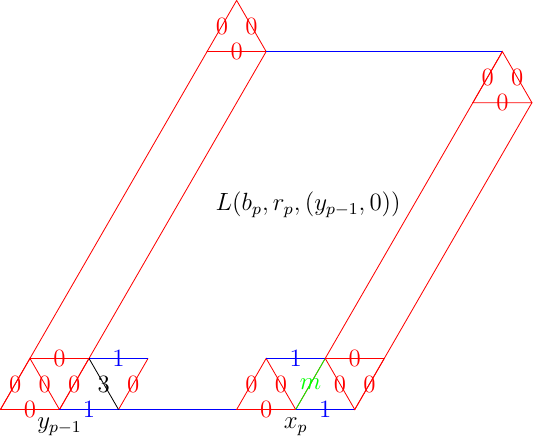}
    \caption{Replacement at $x_p$.}
    \label{fig:group_col_2}
    \end{minipage}%
    \begin{minipage}[c]{0.33\linewidth}
        \includegraphics[scale=0.6]{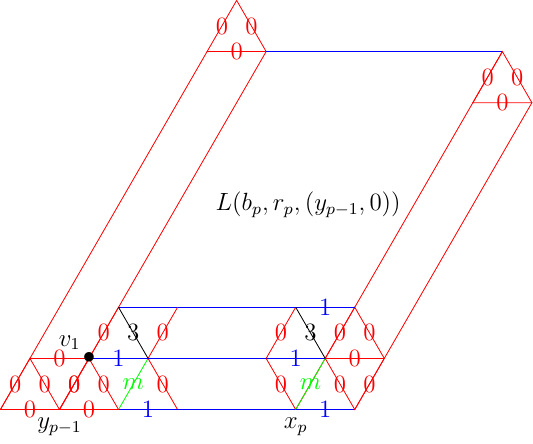}
        \caption{Replacements at $x_p, \dots, y_{p-1}+1$.}
        \label{fig:group_col_3}
        \end{minipage}%
    \end{figure}

\begin{figure} [H]
\begin{minipage}[c]{0.4\linewidth}
\includegraphics[scale=0.6]{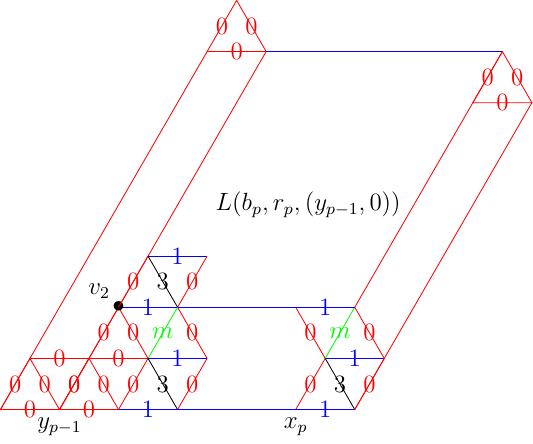}
\caption{First arrow reversal creating an arrow at $v_2$.}
\label{fig:group_col_4}
\end{minipage}%
\begin{minipage}[c]{0.4\linewidth}
    \includegraphics[scale=0.6]{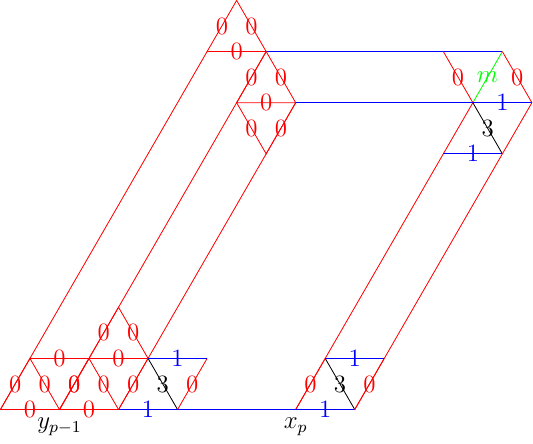}
    \caption{Final configuration $\varphi(C)$.}
    \label{fig:group_col_5}
    \end{minipage}%
\end{figure}


\end{proof}

\noindent
From the previous Lemmas, one derives the following result. 

\begin{proposition}[Edge count during reduction]
\label{prop:reduction_partial_0}
    Let $C$ be a color map such that $G(C, 2) = 0$ and let $b_1$ be defined as in Lemma \ref{lem:bottom_structure}.
    One can reduce $C$ to a color map $C'$ such that $p(C') = 1$ and $b_1(C') = 0$ 
    by removing (respectively adding) $M = n_0n_1 - G(C, 0)$ 
    edges of color $3$ (respectively of color $m$) so that $m(C') = m(C) + M \text{ and } s(C') = s(C) - M$.
\end{proposition}

\begin{proof}
    Apply the transformation of Lemma \ref{lem:grouping_col} until $p=1$. 
    Let us compute the total number $M$ of exchanged $3$ and $m$ colored edges in the process. 
    The only tranformation that changes the number of edges of color $3$ and $m$ is the replacement 
    as in Figure \ref{fig:replacement}. By Lemma \ref{lem:grouping_col}, one has applied 
    \begin{equation*}
        \tilde{M} = r_pb_p + b_{p-1}(r_{p-1}+r_p) + \dots + b_2(r_2 + \dots + r_p) 
        = \sum_{j=2}^p r_j \sum_{i=1}^j b_i - b_1 \sum_{j=2}^p r_j
    \end{equation*}
    replacements of $3$-colored edges by the same number of $m$-colored edges. Moreover, 
    \begin{equation*}
        G(C, 0) =  \sum_{j=1}^p r_j (n_1 - \sum_{i=1}^j b_i) = (n_1 - b_1)\sum_{j=1}^p r_j - \tilde{M} = (n_1 - b_1)n_0 - \tilde{M}.
    \end{equation*}
    If the resulting color map $\tilde{C}$ has $b_1(\tilde{C}) \geq 1$, apply Lemma \ref{lem:grouping_col} a number of times equal to 
    $n_0$ so that the obtained color map $C'$ has $x_1(\tilde{C}) = 0$. This last step removed $b_1(C)n_0 = b_1(\tilde{C})n_0$ edges of 
    color $3$ from $\tilde{C}$ and added the same number of edges of color $m$. Therefore, one has removed 
    $M = \tilde{M} + b_1n_0 = n_0n_1 - G(C, 0)$ edges of color $3$ and added the same number of edges of color $m$.
\end{proof}

\begin{definition}[Reduced color map]
    A color map $C: E_n \rightarrow \{ 0,1,3, m \}$ such that $G(C, 2) = 0$, $p(C) = 1$ and $b_1(C) = 0$
    is called a reduced color map.
\end{definition}

\subsection{Structure of reduced color maps}
\label{subsec:case_p_1}

In this section, we only consider reduced color maps as any color map $C$ such that $G(C, 2) = 0$ 
can be reduced thanks to Proposition \ref{prop:reduction_partial_0} above. We first show that most of the edges in reduced color map 
have their color fixed except in some region, see Figure \ref{fig:region_r}. 
This region consists of specific configuration of edges described in Remark \ref{rem:lozenge_in_R}.
From this we derive the main result of this section in Proposition \ref{prop:nb_m_reduced} which gives the number 
of $m$ and $3$ colored edges in reduced color map. \\
By Lemma \ref{lem:trapeze_filling},
every reduced color map has every edge of its trapeze region $T[n_0, n_0, (0, 0)]$ colored $0$. 
Let 
$R = (E_n \setminus T[n_0, n_0, (0, 0) ]) \cup \{ (n_0 + s\xi^2, n_0 + (s+1) \xi^2), 0 \leq s \leq n_0 -1  \} $ 
be the remaining region pictured in Figure \ref{fig:region_r}.

\begin{figure}[H]
    \centering
    \includegraphics[scale=0.7]{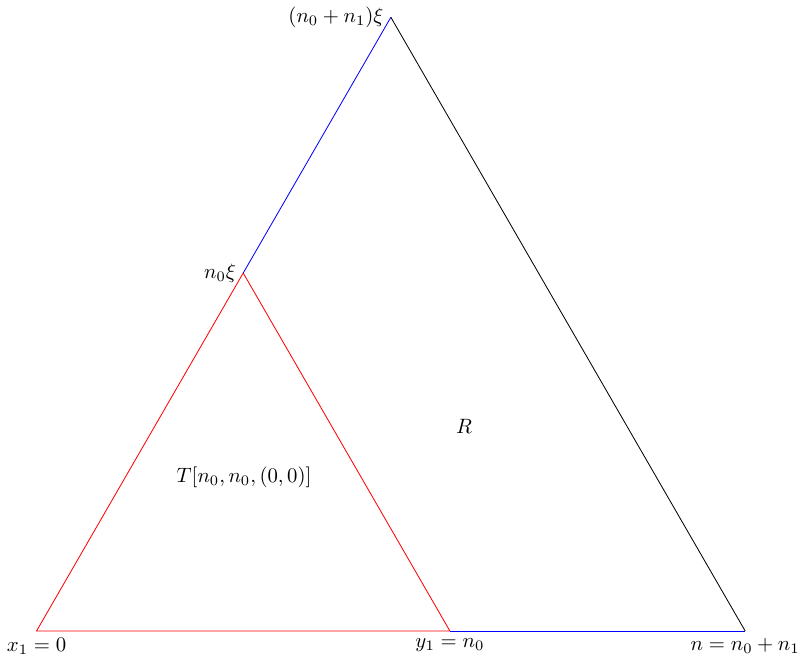}
    \caption{The region $R$ in a reduced color map. Edges outside $R$ have their color fixed.}
    \label{fig:region_r}
\end{figure}

\noindent
The next Lemma shows that edges of color $0$ in $R$ can only be of type $1$.

\begin{lemma}[Edges of color $0$ in $R$]
\label{lem:pieces_in_R}
    Let $R$ be the region above associated a to reduced color map $C$. Then, every edge of color $0$ in $R$ is of type $1$. 
\end{lemma}

\begin{proof}
    Assume for the sake of contradiction that there exists an edge $e^{(0)}$ of type $0$ or $2$ in $R$. 
    Take $e^{(0)}$ such that its origin $x^{(0)}$ has minimal coordinate $x_2$. 
    Since the color map is reduced, edges in $R \cap \partial_0^{(n)}$ have color $1$ so that $x_2 \geq 1$ for edges of type $0$. 
    If $e^{(0)}$ is of type $0$, then one of the edges of type $0$ with origins $x^{(0)} - \xi, x^{(0)} -\xi +1$ 
    or the edge of type $2$ with origin $x^{(0)} - \xi$ is colored $0$. 
    In either case, the minimality of the $x_2$ coordinate of $x^{(0)}$ is violated. 
    If $e^{(0)}$ is of type $2$, the type $2$ edge with center $x^{(0)} + 1$ is colored $0$ and both are 
    opposite edges of a lozenge with middle edge of type $1$ colored $3$ 
    as any other piece would either contradict the minimality of $x_2$ or introduce an edge of type $0$ and color $0$ in $R$.
    Without loss of generality, one can thus assume that $x^{(0)}_0 = 1$. 
    Since $\partial_1^{(n)}$ has edges with colors $0$ or $1$, the upward triangular face containing $e^{(0)}$ would have colors 
    $(0,0,0)$ or $(0,1,m)$ which would either imply that $R$ has an edge of type $0$ and color $0$ or 
    contradict the minimality of $x_2$.
\end{proof}


\begin{remark}[Lozenges in $R$]
\label{rem:lozenge_in_R}
    The two opposite $0$ colored edges of a either $3$ or $m$ lozenge have the same type. 
    Since Lemma \ref{lem:pieces_in_R} shows that $0$ colored edges in $R$ have type $1$, 
    the only orientations of $3$ and $m$ lozenges in $R$ are such that the $3$ edge has type $0$ and the $m$ edge has type $2$, 
    see Figure \ref{fig:region_R_loz}.  
\end{remark}

\begin{figure}[H]
    \centering
    \includegraphics[scale=1]{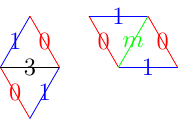}
    \caption{The two possible orientations of $3$ and $m$ lozenges in the region $R$.}
    \label{fig:region_R_loz}
\end{figure}

\noindent
We know from Lemma \ref{lem:pieces_in_R} that the color map $C$ in $R$ consists in triangular faces $f$ such that $C(f) = (1,1,1)$ 
together with either $3$ or $m$ lozenges oriented as in Figure \ref{fig:region_R_loz}. 
In the rest of this section, we will view the color map $C$ on $R$ as a configuration of paths of color $0$ 
from $R \cap \partial T[n_0, n_0, (0, 0)]$ to $R \cap \partial_2 C$ as follows. \\ 
To each $3$ or $m$ lozenge of Figure \ref{fig:region_R_loz}, associate a line segment by joining the two centers of the opposite $0$ 
colored edges. 
We define paths $(p_i, 1 \leq i \leq n_0)$ simultaneously. For each $1 \leq i \leq n_0$, the path $p_i$ starts in the middle of the 
edge of type $1$ in $R \cap \partial T[n_0, n_0, (0, 0)]$ with origin $(n_0, 0) + i \xi^2$. 
At each step, a path having its endpoint with coordinate $x_2$ is continued by the affine line segment in the adjacent lozenge in $R$ 
having an edge of color $0$ with origin of coordinate $x_2-1$.
Since this lozenge can only be one of the two lozenges of Figure \ref{fig:region_R_loz}, the paths are non-intersecting. 
After $n_1$ steps, endpoints are located in the middle of type $1$ edges of color $0$ on $R \cap \partial_2 C$. 
As there are $n_0$ such edges, the paths $(p_1, \dots, p_{n_0})$ form a set of non-intersecting paths 
where for each $1 \leq i \leq n_0$, the path $p_i$ has origin at 
$o_i = (n_0, 0) + (i-\frac{1}{2}) \xi^2$ and target $t_i = o(e_i) + \frac{1}{2} \xi^5$ where 
$e_1, \dots e_{n_0}$ are the $0$ colored edges on $\partial_1 C$ ordered such that $h(e_1) > h(e_2) > \dots > h(e_{n_0})$. \\
The paths $(p_i, 1 \leq i \leq n_0)$ can have two possible steps. 
We call the step induced by a $m$ lozenge a horizontal step and the step induced by a $3$ lozenge a vertical step in accordance 
with the red line segment joining the two opposite $0$ colored edges in the lozenges of Figure \ref{fig:region_R_loz}.

\begin{proposition}[Number of $m$ edges in reduced color maps]
\label{prop:nb_m_reduced}
    Let $C$ be a reduced color map. 
    Denote by $n(m, R)$ and $n(3, R)$ the respective number of horizontal and vertical steps in $R$. 
    Recall that $m(C)$ and $s(C)$ respectively denote the number of $m$ lozenges and $3$ lozenges in $C$. 
    Then, 
    \begin{equation}
    \label{eq:nb_hc_R}
        m(C) = n(m, R) = G(C, 1)
    \end{equation}
    and 
    \begin{equation}
    \label{eq:nb_sc_R}
        s(C) = n(3, R) = n_0n_1 - G(C, 1).
    \end{equation}
\end{proposition}

\begin{proof}
    The only region in $E_n$ where a reduced color map $C$ has $m$ colored edges is $R$. This amounts to count 
    the number of horizontal steps in any path configuration $(p_i, 1 \leq i \leq n_0)$. Recall that $e_1, \dots e_{n_0}$ are the 
    $0$ colored edges on $\partial_1 C$ ordered such that $h(e_1) > h(e_2) > \dots > h(e_{n_0})$. As each path $p_i$ goes from
    \begin{equation*}
        o_i = (n_0, 0) + (i-\frac{1}{2}) \xi^2 = \left( n_0 - \frac{1}{2}(i - \frac{1}{2}), (i - \frac{1}{2}) \frac{\sqrt{3}}{2} \right)
    \end{equation*}
    to 
    \begin{equation*}
        t_i = o(e_i) + \frac{1}{2} \xi^5 = (n, 0) + (n - h(e_i)) \xi^2 + \frac{1}{2} \xi^5 = 
        \left( n - \frac{1}{2} (n-h(e_i) - \frac{1}{2}), \frac{\sqrt{3}}{2} (n - h(e_i) - \frac{1}{2}) \right),
    \end{equation*}
    the number of vertical steps in $p_i$ is given by
    \begin{equation*}
        \frac{2}{\sqrt{3}} \cdot \left( \frac{\sqrt{3}}{2} (n - h(e_i) - \frac{1}{2}) - (i - \frac{1}{2}) \frac{\sqrt{3}}{2} \right) = n - h(e_i) - i 
    \end{equation*}
    so that the total number of vertical steps in the path configuration is 
    \begin{equation*}
        n(3, R) = \sum_{i = 1}^{n_0} (n - h(e_i) - i).
    \end{equation*}
    Moreover, the number of $1$ colored edges $e'$ such that $h(e') < h(e_i)$ is given by
    \begin{equation*}
        n(C, e_i) = n_1 - (n - h(e_i) - i)
    \end{equation*}
    so that
    \begin{equation*}
        G(C, 1) = \sum_{i = 1}^{n_0} n(C, e_i) = \sum_{i = 1}^{n_0} (n_1 - (n - h(e_i) - i )) = n_0n_1 - n(3, R). 
    \end{equation*}
    Therefore, 
    \begin{equation*}
        n(3, R) = n_0n_1 - G(C, 1)
    \end{equation*}
    and, since each of the $n_0$ paths has a total number of steps given by $n_1$,
    \begin{equation*}
        n(m, R) = n_0n_1 - n(3, R) = G(C, 1).
    \end{equation*}
\end{proof}

\begin{remark}[Number of reduced color maps]
    One can derive the number of reduced color maps $N_{red}(n_0, n_1)$ since this number is equal to the number of paths 
    configurations $(p_i, 1 \leq i \leq n_0)$ which can be computed by the determinantal 
    formula of Lindström, Gessel and Viennot \cite{Lindstrom}, \cite{Gessel_Viennot}:
    \begin{equation}
        N_{red}(n_0, n_1) = \det A
    \end{equation}
    where $A = (a_{i,j}, 1 \leq i,j \leq n_0)$ is the matrix whose coefficients are given by 
    \begin{equation}
        a_{i,j} = \binom{n_1}{ n - h(e_j) - i }
    \end{equation}
    with the convention that $a_{i,j} = 0$ if $h(e_j) + i \geq n$.
\end{remark}

\subsection{Proof of Theorem \ref{th:nb_hard_crossings} in the case $G(C, 2) = 0$}
\label{subsec:proof_g_zero}

Let $C$ be a color map such that $G(C, 2) = 0$. 
Thanks to Proposition \ref{prop:reduction_partial_0}, one can reduce $C$ to a reduced color map $C'$ such that 
\begin{equation}
    m(C') = m(C) + n_0n_1 - G(C, 0)
\end{equation}
and 
\begin{equation}
    s(C') = s(C) - n_0n_1 + G(C, 0)
\end{equation}
Using \eqref{eq:nb_hc_R} and \eqref{eq:nb_sc_R} for the reduced color map $C'$, 
\begin{align}
    m(C') &= G(C', 1) \\
    s(C') &= n_0n_1 - G(C', 1).
\end{align}
Moreover, the reduction $C \mapsto C'$ does not change $\partial_1 C = \partial_1 C'$ so that $G(C, 1) = G(C', 1)$. 
Thus,
\begin{equation}
    m(C) = G(C, 0) + G(C, 1) - n_0n_1
 \end{equation}
which is \eqref{eq:hc_count_general} since $G(C, 2) = 0$ and
\begin{equation}
        s(C) = s(C') - G(C, 0) + n_0n_1 = 2n_0n_1 - G(C, 0) - G(C, 1) = n_0n_1 - m(C)
\end{equation}
which is \eqref{eq:sc_count_general}.

\section{The general case}
\label{sec:general_case}

In this section, we prove \eqref{eq:hc_count_general} by induction on $G(C, 2)$. 
The case where $G(C, 2) = 0$ has been treated in Section \ref{subsec:proof_g_zero}. 
We first introduce a procedure in Section \ref{subsec:gash_prop} which takes a color map $C$ for which $G(C, 2) \geq 1$ 
and transform it to a color map $C'$ such that $G(C', 2) = G(C, 2) - 1$. Using the previous transform, we finish the proof of 
Theorem \ref{th:nb_hard_crossings} in Section  \ref{subsec:proof_general}.

\subsection{Gash propagation}
\label{subsec:gash_prop}

We introduce local configuration of two edges of the same type sharing a vertex called a gash, 
see Definition \ref{def:gash}, which is inspired from gashes previously defined in \cite{puzzle_conj_two_step}, 
\cite{Buch_mutations} and \cite{françois2024positiveformulaproductconjugacy}. Gashes will propagate across a color map $C$ 
by local rules presented in Definition \ref{def:gash_propag_2} until reaching some prescribed configuration or hitting 
$\partial_1^{(n)}$.

\begin{definition}[Gash]
\label{def:gash}
    Let $x \in T_n$. A gash with center $x$ is the union of the two edges $(x, x - \xi^{2l}), (x + \xi^{2l}, x)$ for $l \in \{1,2 \}$ with the data of 
    \begin{enumerate}
        \item Original colors given by   \begin{align*}
        C((x, x - \xi^{2l})) = 1, \ C((x + \xi^{2l}, x)) = 0 &\text{ if } l = 1 \\
        C((x, x - \xi^{2l})) = 0, \ C((x + \xi^{2l}, x)) = 1 &\text{ if } l = 2.
    \end{align*}
    \item New colors given by replacing $0$ with $1$ and vice-versa in original colors.
    \end{enumerate}
    The type of a gash is defined as the type $l \in \{1,2 \}$ of its edges.
\end{definition}

\noindent
Let $g$ be gash of type $2$. The only possible values of the color map $C$ adjacent to $g$ are given by the configurations of Figure 
\ref{fig:gash_configurations} that we label from $(i)$ to $(vi)$. The configurations can be rotated for type $1$ gashes. 

\begin{figure}[H]
    \centering
    \includegraphics[scale=0.8]{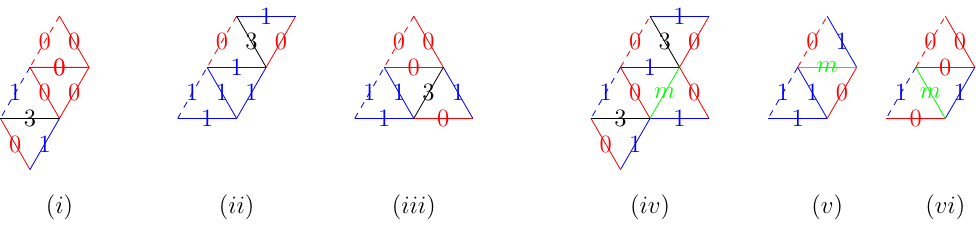}
    \caption{Possible  adjacent configuration to a gash of type $2$ in dashed edges. Only the original colors of gashes are represented.}
    \label{fig:gash_configurations}
\end{figure}

\begin{definition}[Gash propagation]
\label{def:gash_propag_2}
    Let $g$ be gash of type $2$ with center $x$ adjacent to a configuration $(i), (ii) $ or $(iii)$. We define the propagation 
    of $g$ to be the gash $g'$ having center $x + \overline{\xi}$ (resp. $x+1$) in the case of configuration 
    $(i)$ (resp. $(ii)$ or $(iii)$) together with the local replacement of Figure \ref{fig:gash_prop} depending on the adjacent configuration.
\end{definition}

\begin{figure}[ht]
    \centering
    \includegraphics[scale=0.8]{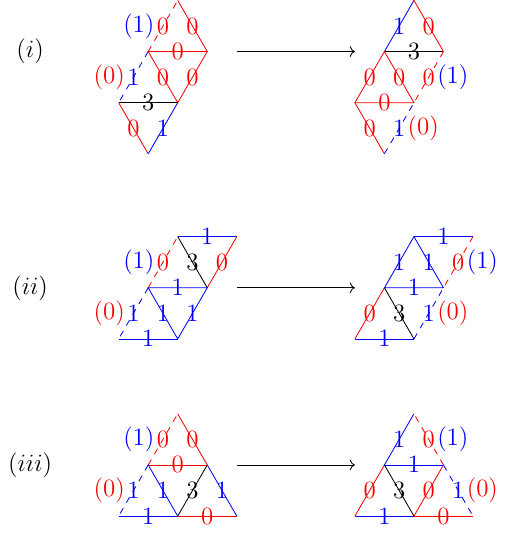}
    \caption{Propagation of a gash through configurations $(i)$, $(ii)$ and $(iii)$. New colors of gashes are written in parenthesis.}
    \label{fig:gash_prop}
\end{figure}

\noindent
If $g$ is adjacent to a configuration $(iv)$, notice that there is a $0$ opening at its center $x$ and thus an arrow of color $0$ at $x$ 
with type $0$. Reverting this arrow yields a configuration $(i)$ adjacent to $g$ and we define the propagation of $g$ to be the gash $g'$ 
of the same type as in step $(1)$. Using a rotation, one defines propagations for type $1$ gashes with the exception of configuration $(iii)$ 
where the propagated gash is the gash of type $2$ with center $x' = x+1$.

\begin{definition}[Propagation algorithm]
\label{def:propagation_algorithm}
    Define the following algorithm.
    \begin{enumerate}
    \item[\textbf{Input:}] A color map $C$ and a gash $g$ of type $l \in \{1,2\}$.
    \item Set $g^{(0)} = g$, $x^{(0)} = x(g)$, $t^{(0)} = t(g)$.
        \item \textbf{WHILE} $g^{(s)}$ is adjacent to $(i), (ii), (iii)$ or $(iv)$: set $g^{(s+1)}$ to be the propagation of $g^{(s)}$ with 
        center $x^{(s+1)}$ and type $t^{(s+1)}$.
    \end{enumerate}
\end{definition}

\begin{proposition}[Gash propagation]
\label{prop:gash_propagation}
    Let $g$ be a gash of type $2$ on $\partial_2^{(n)}$. The propagation algorithm terminates at a gash $g'$ adjacent to configuration 
    of type $(v)$, $(vi)$ or on $ \partial_n^{(1)} $. 
\end{proposition}

\begin{proof}
    One checks that propagations of Definition \ref{def:gash_propag_2} do not change the type of the gash except 
    in the case of a configuration $(iii)$ which turns a gash of type $2$ into a gash of type $1$ and vice-versa. 
    Since the starting gash $g$ has type $2$, the gashes have type either $1$ or $2$ along the propagation. 
    At each step of the gash propagation, one has either $ x^{(s+1)}_0 < x^{(s)}_0$ or $x^{(s+1)}_1 > x^{(s)}_1$ and 
    $t^{(s)} \in \{ 1,2 \}$ which implies that the while loop terminates on a gash $g^{(\infty)}$ which is adjacent 
    to a configuration $(v)$ or $(vi)$ or necessarily of type $1$ on $\partial_n^{(1)}$.
\end{proof}

\noindent
In the case where a gash is adjacent to a configuration $(v)$ or $(vi)$, one still wants to replace the original colors by the new ones. 
To do so, we introduce a local transformation called gash removal in Definition \ref{def:gash_removal}.

\begin{definition}[Gash removal in $(v)$ and $(vi)$]
\label{def:gash_removal}
    Let $C$ be a color map and let $g$ be a gash of type $2$ with center $x$ adjacent to a configuration $(v)$ or $(vi)$. 
    The removal of $g$ is the new color map $C'$ defined by
    \begin{align*}
        C'((x, x - \xi^4)) &= 1, C'((x + \xi^4, x)) = 0 \\
        C'((x+1, x)) &= 1, C'((x, x - \xi^2)) = 3 \text{ if } C((x+1, x)) = m \\
        C'((x+1, x)) &= 3, C'((x, x - \xi^2)) = 0 \text{ if } C((x+1, x)) = 0 \\
        C'(e) &= C(e) \text{ otherwise.} 
    \end{align*}
\end{definition}

\noindent
See Figure \ref{fig:gash_removal} for an illustration. Using rotation, one defines the gash removal for type $1$ gashes adjacent to a configuration $(v)$ or $(vi)$. 

\begin{figure}[H]
    \centering
    \includegraphics[scale=0.8]{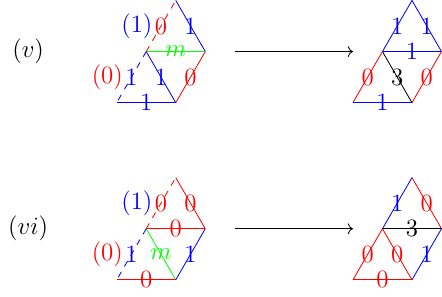}
    \caption{Gash removal in configurations $(v)$ and $(vi)$.}
    \label{fig:gash_removal}
\end{figure}

\subsection{Proof of Theorem \ref{th:nb_hard_crossings}}
\label{subsec:proof_general}

We are now in position to prove Theorem \ref{th:nb_hard_crossings} by induction on $G(C, 2)$. 
The case  $G(C, 2) = 0$ has been treated in Section \ref{subsec:proof_g_zero}. 
Assume that the identities \eqref{eq:hc_count_general} and \eqref{eq:sc_count_general} hold 
for color maps $C$ such that $G(C, 2) \leq N$ and consider a color map $C$ such that $G(C, 2) = N+1$. 

Since $G(C, 2) \geq 1$, there exists a pair of edges $(x, x - \xi^{4}), (x + \xi^{4}, x) \in (\partial_n^{(2)})^2$ 
such that $C((x, x - \xi^{4})) = 0, \ C((x + \xi^{4}, x)) = 1$ for some $x \in T_n$. Let $g$ be the gash on 
$\partial_n^{(2)}$ with center $x$, original colors as above and new colors given by 
$C((x, x - \xi^{4})) = 1, \ C((x + \xi^{4}, x)) = 0$ as in Definition \ref{def:gash}. 
Applying the propagation algorithm of Definition \ref{def:propagation_algorithm} and using 
Proposition \ref{prop:gash_propagation} yields a gash $g'$ adjacent to a configuration $(v)$, $(vi)$ 
or on $\partial_n^{(1)}$. In the case of configurations $(v)$ or $(vi)$, apply the gash removal of 
Definition \ref{def:gash_removal}. In the case where $g' \in \partial_n^{(1)}$, replace the original 
colors by the new ones so that the $0$ and $1$ colors are swapped. Call $C'$ the resulting color map. Then, 
\begin{align}
    G(C', 2) &= G(C, 2) - 1 = N \\
     G(C', 0) &= G(C, 0)
\end{align}

\noindent
In the case where one used gash removal, 
\begin{equation}
    G(C', 1) = G(C, 1), \ \ m(C') = m(C) - 1, \text{ and } s(C') = s(C) + 1.
\end{equation}
whereas in the case where $g' \in \partial_n^{(1)}$,
\begin{equation}
    G(C', 1) = G(C, 1) +1, \ \ m(C') = m(C), \text{ and } s(C') = s(C).
\end{equation}
In both cases, applying the induction hypothesis to $C'$ gives 
\begin{align}
    m(C') &= G(C', 0) + G(C', 1) + G(C', 2) - n_0n_1 \\
    s(C') &= 2n_0n_1 - G(C', 0) - G(C', 1) - G(C', 2)
\end{align}
which gives
\begin{align}
    m(C) &= G(C, 0) + G(C, 1)  + G(C, 2) - n_0n_1 \\
    s(C) &= 2n_0n_1 - G(C, 0) - G(C, 1) - G(C, 2)
\end{align}
as desired.
\\
\\
We finally state a Corollary of the main result which counts the number of faces having all of their edges of the same color, either $0$ or $1$.

\begin{corollary}[Number of triangular pieces]
\label{cor:triangular_faces_count}
 Let $C$ be a color map. Let $n^{(j)}, n_{(j)}, j \in \{ 0, 1 \}$ denote respectively the number of direct and reverse triangular faces $f \in F_n$ having all their edges of color $j$. Then,
    \begin{equation}
    \label{eq:nb_triangular_faces}
        n^{(j)} = \frac{n_j(n_j+1)}{2} \text{ and } n_{(j)} = \frac{n_j(n_j-1)}{2}.
    \end{equation}
\end{corollary}

\begin{proof}
    One can check that the gash propagation and removal steps preserve the number of faces $f$ having edge colors $(0,0,0)$ or $(1,1,1)$. 
    It therefore suffices to count them in a reduced color map which gives \eqref{eq:nb_triangular_faces}.
\end{proof}

\printbibliography

\end{document}